\documentclass[12pt]{amsart}

\usepackage{amsmath, amssymb, amsthm, a4wide}
\usepackage{color}
\usepackage{graphicx}

\newtheorem{thm}{Theorem}[section]
\newtheorem{lem}[thm]{Lemma}

\newtheorem{cor}[thm]{Corollary}

 \def\d{{\rm d}}
 
 \def\R{{\mathbb R}}

\title[Minimal periods of ODEs in Banach spaces]{Minimal periods for ordinary differential equations in strictly convex Banach spaces and explicit bounds for some $\ell^p$-spaces}
 \author{Michaela A. C. Nieuwenhuis}
 \address{OxPDE, Mathematical Institute, Radcliffe Observatory Quarter, Woodstock Road, Oxford OX2 6GG, UK.}
 \email{Michaela.Nieuwenhuis@maths.ox.ac.uk}

 \author{James C. Robinson}
 \address{Mathematics Institute, Zeeman Building, University of Warwick, Coventry CV4 7AL, UK.}
 \email{J.C.Robinson@warwick.ac.uk}

\author{Stefan Steinerberger}
\address{Mathematisches Institut, Universit\"at Bonn, Endenicher Allee 60, 53115 Bonn, Germany}
\email{steinerb@math.uni-bonn.de}

\begin{document}
\begin{abstract}
Let $ x(t)$ be a non-constant $T$-periodic solution to the ordinary differential equation $\dot{ x}=f(x)$ in a Banach space $X$, where $f$ is assumed
to be Lipschitz continuous with constant $L$. Then there exists a constant $c$ such that $TL\geq c$, with $c$ only depending on $X$. It is known that
$c\ge 6$ in any Banach space and that $c=2\pi$ in any Hilbert space, but whereas the bound of $c=2\pi$ is sharp in any Hilbert space,
there exists only one known example of a Banach space such that $c=6$ is optimal. In this paper, we show that the inequality $TL\ge 6$ is in fact strict in any strictly convex Banach space.
Moreover, we improve the lower bound for $\ell^p(\mathbb{R}^n)$ and $L^p(M,\mu)$ for a range of $p$ close to $p=2$ by using a form of Wirtinger's inequality for functions in $W^{1,p}_{per}([0,T],L^p(M,\mu))$.
\end{abstract}


\maketitle

\section{Introduction}
Consider the ordinary differential equation $\dot{x}=f(x)$ in a Banach space $X$, where $f$ is Lipschitz continuous with constant $L$, that is for any $x,y\in X$
$$
\|f(x)-f(y)\|_X\leq L\|x-y\|_X.
$$
In this case one can relate the period $T$ of any non-constant periodic orbit to the Lipschitz constant $L$ via the inequality $TL\ge c$.
In 1969, Yorke \cite{yor} proved that $c=2\pi$ when $X=\mathbb{R}^n$ with its usual norm. Lasota \& Yorke \cite{lasyor} showed that the proof extends to arbitrary Hilbert spaces and they
proved the bound $c = 4$ for any Banach space. This was improved to $c=4.5$ by Busenberg \& Martelli \cite{busmar} and finally to $c=6$ by Busenberg, Fisher \& Martelli \cite{busfishmar}
who also gave another proof for $c=2\pi$ in any Hilbert space using Wirtinger's inequality. An obvious extension of the simple two-dimensional example
$$
\dot x=Ly\qquad\dot y=-Lx
$$
shows that $c=2\pi$ is sharp in any Hilbert space. Busenberg, Fisher \& Martelli \cite{busfishmar2} also constructed an example of an ODE on a periodic orbit of period $1$, which when viewed as a subset of
$L^1([0,1]^2)$ has Lipschitz constant $L=6$, showing that $c=6$ is sharp for general Banach spaces.\\

However, some interesting questions about minimal periods remain unanswered. Does there exist an ODE in a finite-dimensional Banach space such
that the lower bound of $TL=6$ is obtained? Does $TL\geq 2\pi$ characterise Hilbert spaces? Is there a (non-Hilbert) Banach space for which $c> 6$? 

The results in this paper address this last question. First we show that in strictly convex Banach spaces necessarily $TL>6$. For these normed topological vector spaces the unit ball is a strictly
convex set. It is easy to see that the unit balls in $\ell^1$ and $\ell^\infty$ contain a line segment and are therefore not strictly convex
sets whereas the unit balls for all $1<p<\infty$ are strictly convex. This result nicely complements the current theory because the only example
for a Banach space with $c=6$ is $L^1$.\\

However, we prove not only that the inequality is strict in any strictly convex Banach space but we are also able to push the bound a little further
for the simplest family of interesting finite-dimensional Banach spaces, namely $\ell^p(\R^n)$, that is $\mathbb{R}^n$ equipped with the $\ell^p$-norm,
$$
\|(x_1,\ldots,x_n)\|_{\ell^p}=\left(\sum_{j=1}^n|x_j|^p\right)^{1/p}.
$$

It is remarkable that even for Euclidean spaces with the family of $\ell^p$-norms the optimal constant is not known\footnote{Unfortunately there appears to be an error in one of the integral calculations in the paper by Zevin \cite{zev} which claims to show that $c=2\pi$ in $\ell^\infty(\R^n$).} for $p\neq 2$.
Our second contribution in this paper is to point out that by using a generalised form of Wirtinger's inequality, one can find explicit bounds on $c$ which are strictly larger than $6$ in a range of $\ell^p$-spaces near $p=2$ ($1.43\lesssim p\lesssim 3.35$). A similar argument also works in the infinite-dimensional Lebesgue spaces $L^p(M,\mu)$.

 We should mention the interesting related result, due to Zevin \cite{zev2}, that if $X$ is a finite-dimensional Banach space and one considers the second order equation $\ddot x=f(x)$ with $f:X\to X$ Lipschitz with constant $L^2$, then $TL\ge 2\pi$ independent of the space $X$. (The paper \cite{zev2} claims a similar result for the first order equation $\dot x=f(x)$, but there is a small error in the proof of equation (11). Nevertheless, Zevin's argument readily yields the result we have stated for $\ddot x=f(x)$.)

\section{Minimal periods in strictly convex Banach spaces}
Let us start this section by stating the main result of this paper:
\begin{thm}\label{main result}
Let $X$ be a strictly convex Banach space. Then
$$ TL > 6.$$
\end{thm}

In fact the proof of this statement is a refinement of an integral inequality originally introduced by Busenberg, Martelli \& Fisher \cite{busfishmar}. The revised
version of the result is summarised in the following lemma.
\begin{lem}
\label{BanachLemma}
Let $X$ be a normed space and $y:\mathbb{R}\rightarrow X$ be a continuous, $T$-periodic function such that $ \left\Vert\dot{y}(t)\right\Vert$ is integrable. Then $$\int_0^T \int_0^T \left\Vert y(t)-y(s) \right\Vert \d s\; \d t \leq \frac{T}{6} \int_0^T \int_0^T  \left\Vert\dot{y}(t)-\dot{y}(s)\right\Vert \d s \;\d t.$$
If $X$ is a strictly convex Banach space, then the above inequality is in fact strict.
\end{lem}

Before we go into details of the proof, we show how Busenberg, Fisher \& Martelli used it to establish $TL\geq 6$ for any Banach space.

\begin{proof}[Proof of Theorem \ref{main result}]
Applying Lemma \ref{BanachLemma} and using the Lipschitz continuity of $f$, it follows that
\begin{eqnarray*}
\int_0^T\int_0^T || x(t)- x(s)||\d s\;\d t &\leq & \frac{T}{6} \int_0^T\int_0^T ||\dot{ x}(t)-\dot{ x}(s)||\d s\;\d t\\
&=& \frac{T}{6}\int_0^T\int_0^T ||f( x(t))-f( x(s))||\d s\;\d t\\ &\leq & \frac{LT}{6}\int_0^T \int_0^T || x(t)- x(s)||\d s\;\d t.
\end{eqnarray*}
Dividing both sides of the inequality by $\int_0^T\int_0^T || x(t)- x(s)||\d s\;\d t$ yields the result.
\end{proof}

We now turn to the main proof of this section.

\begin{proof}[Proof of Lemma \ref{BanachLemma}]\label{Integral Lemma Martelli}
We know that $y$ is periodic with period $T$. Hence its integral over one period is shift invariant and thus $$\int_0^T \int_0^T \left\Vert y(t+s)-y(s)\right\Vert \d s\d t=\int_0^T \int_0^T \left\Vert  y(t)-y(s)\right\Vert \d s\d t.$$
Using the above observation, we can derive the following integral expression

\begin{eqnarray}\label{inequality martelli}
\int_0^T \hspace{-0.2cm} \int_0^T  \hspace{-0.1cm}\left\Vert y(t)-y(s)\right\Vert \d s\;\d t \hspace{-0.2cm} \nonumber
&=&\hspace{-0.3cm}\int_0^T \hspace{-0.2cm} \int_0^T  \hspace{-0.05cm}\left\Vert y(t+s)-y(s)\right\Vert \d s\; \d t\nonumber\\
&=&\hspace{-0.3cm}\int_0^T \hspace{-0.2cm}\int_0^T \hspace{-0.05cm}\frac{(T-t)t}{T} \left\Vert\frac{y(t+s)-y(s)}{t}-\frac{y(s)-y(s+t-T)}{T-t}\right\Vert \d s\; \d t\nonumber\\
&=&\hspace{-0.3cm}\int_0^T \hspace{-0.2cm}\int_0^T \frac{(T-t)t}{T^2}\left\Vert \int_0^{T} \dot{y}\left(s+\frac{tr}{T}\right)- \dot{y}\left(s+\frac{tr}{T}-r\right)\d r\right\Vert \d s\; \d t\nonumber\\
&\leq &\hspace{-0.3cm}\int_0^T \hspace{-0.2cm} \int_0^T \frac{(T-t)t}{T^2}\int_0^T \left\Vert \dot{y}\left(s+\frac{tr}{T}\right)-\dot{y}\left(s+\frac{tr}{T}-r\right)\right\Vert \d r\;\d s\; \d t\\
&= &\hspace{-0.3cm}\int_0^T \hspace{-0.05cm} \frac{(T-t)t}{T^2}\int_0^T \hspace{-0.2cm}\int_0^T \left\Vert \dot{y}\left(s+\frac{tr}{T}\right)-\dot{y}\left(s+\frac{tr}{T}-r\right)\right\Vert \d s\; \d r\; \d t\nonumber
\end{eqnarray}
The last inner integral has been taken over one period, so we may shift it by $tr/T$ in order to obtain
\begin{eqnarray*}
\int_0^T\int_0^T \left\Vert y(r)-y(s)\right\Vert \d r\;\d s &\leq & \int_0^T \frac{(T-t)t}{T^2}dt \int_0^T \int_0^T \left \Vert \dot{y}(s+r)-\dot{y}(s)\right \Vert \d s\; \d r\\
&=& \frac{T}{6} \int_0^T \int_0^T \left \Vert \dot{y}(r)-\dot{y}(s)\right \Vert \d s\; \d r
\end{eqnarray*}
giving us the desired inequality for arbitrary Banach spaces.\\

From now on we consider the case when $X$ is in fact a strictly convex Banach space. The only actual inequality in the above argument occurs in line (\ref{inequality martelli})
where we use the triangle inequality for the Banach space $X$. Note that in doing so, we have a weight
$$\frac{(T-t)t}{T^2}$$
in front of the inner integral which vanishes at $t = 0,T$. In particular, if we show that this inequality actually has to be strict for some $s$ and some $0 < t < T$, our statement follows. Additionally, because of the weight, these conditions are tight as the triangle inequality could fail to be strict at $t = 0,T$ without causing the chain of inequalities to become strict. Note that from the continuity of $\dot{y}(t)$ we obtain that the functions
$$ (s,t) \rightarrow \left\|\int_{0}^{T}\dot{y}\left(s+\frac{tr}{T}\right) -\dot{y}\left(s+\frac{tr}{T}-r\right)\d r\right\|$$
$$ (s,t) \rightarrow \int_{0}^{T}\left\|\dot{y}\left(s+\frac{tr}{T}\right) -\dot{y}\left(s+\frac{tr}{T}-r\right)\right\| \d r$$
are continuous as well. Fix $s$ and $0 < t < T$, fix an arbitrarily fine decomposition $0 = a_0 < a_1 < \dots < a_n = T$ and abbreviate
$$ b_i := \int_{a_i}^{a_{i+1}}\dot{y}\left(s+\frac{tr}{T}\right) \qquad \mbox{and} \qquad c_i := \int_{a_i}^{a_{i+1}}\dot{y}\left(s+\frac{tr}{T}-r\right).$$

If there is in fact equality in (1), then we need to have equality in every step of iteratively applying the triangle inequality and thus
\begin{eqnarray*}
\left\|\sum_{i=0}^{n-1}{b_i-c_i}\right\| &=& \left\|b_0 - c_0\right\| + \left\|\sum_{i=1}^{n-1}{b_i-c_i}\right\| \\
&=& \left\|b_0 - c_0\right\| + \left\|b_1 - c_1\right\| + \left\|\sum_{i=2}^{n-1}{b_i-c_i}\right\| \\
&=& \dots \\
&=& \sum_{i=0}^{n-1}{\left\|b_i-c_i\right\|}.
\end{eqnarray*}
W.l.o.g. we assume that all the terms satisfy $b_{i}-c_{i} \neq \textbf{0}$. Strict convexity implies in the last line of this argument that
$b_{n-2} - c_{n-2}$ and $b_{n-1} - c_{n-1}$ are collinear. By the same reasoning $b_{n-3} - c_{n-3}$ and $(b_{n-2} - c_{n-2}) + (b_{n-1} - c_{n-1})$
are collinear, however, the last expression itself is collinear to $b_{n-2}-c_{n-2}$ as well as $b_{n-1}- c_{n-1}$. Iterating this argument shows
that all $b_{i}-c_{i}$ are necessarily collinear. Using the continuity of $\dot{y}(t)$, making the partition sufficiently small and applying the
fundamental theorem of calculus, we can deduce that for every fixed $s$ and $0 < t < T$ there exists a vector $\textbf{v} \in X$ and a function
$g:[0,T] \rightarrow \mathbb{R}$ such that for all $0 \leq r \leq T$
\begin{equation}\dot{y}\left(s+\frac{tr}{T}\right) - \dot{y}\left(s+\frac{tr}{T}-r\right) = g(r)\textbf{v}. \label{fund}\end{equation}
 Note, however, that both $g$ and $\textbf{v}$ depend on the previously fixed $s,t$. Since $y$ is not constant, it is possible to find and fix an $s$ such that
$$ \dot y(s) \neq \textbf{0}.$$
We now claim that this already implies that for all $0\leq r\leq T$
$$\dot{y}(s+r)  = \tilde{g}(r)\textbf{v} + \dot{y}(s).$$
Suppose this was false, then there is an $r$ such that
$$ \dot{y}(s+r) \notin \left\{\dot y(s) + \lambda \textbf{v}| \lambda \in \mathbb{R}\right\}.$$
In particular,
$$ \min_{\lambda \in \mathbb{R}}{\|\dot{y}(s+r) - \dot y(s) + \lambda \textbf{v}\|} > 0.$$
This, however, can be seen to contradict (\ref{fund}) by taking $t$ sufficiently small.\\

\noindent Since $y$ is periodic with period $T$,
$$ \int_{0}^{T}{\dot{y}(s+r)\d r} = \textbf{0} = \left(\int_{0}^{T}{\tilde{g}(r)\d r}\right)\textbf{v} + T\dot{y}(s).$$
This implies that $\dot{y}(s)$ is a scalar multiple of $\textbf{v}$, in which case
$$ \dot{y}(s+r) = \left(\tilde{g}(r) - \frac{1}{T}\int_{0}^{T}{\tilde{g}(r)\d r}\right)\textbf{v}.$$
This establishes that $\dot y(t)$ is one-dimensional, that is
$$ \dot y(t) = h(t) \textbf{v}$$
for some $\textbf{v} \neq \textbf{0}$ and a continuous, $T-$periodic function $h:[0,T] \rightarrow \mathbb{R}$.\\

Going back to an earlier stage of the argument, we had that for any fixed $s$ and $0<t<T$ the application of
the triangle inequality needs to be strict, that is
$$ \left\|\int_{0}^{T}\dot{y}\left(s+\frac{tr}{T}\right) -\dot{y}\left(s+\frac{tr}{T}-r\right)\d r\right\| = \int_{0}^{T}\left\|\dot{y}\left(s+\frac{tr}{T}\right) -\dot{y}\left(s+\frac{tr}{T}-r\right)\right\| \d r.$$
Plugging in the relation $ \dot y(t) = h(t) \textbf{v}$,
we require that for any fixed $s,t$ with $0 < t < T$
$$ \left| \int_{0}^{T}{h\left(s+\frac{tr}{T}\right) - h\left(s+\frac{tr}{T}-r\right)\d r} \right| = \int_{0}^{T}{\left| h\left(s+\frac{tr}{T}\right) - h\left(s+\frac{tr}{T}-r\right)\right| \d r}.$$
However, since $h$ is continuous and
$$ \int_{0}^{T}{h(z)\d z} = 0,$$
$h$ has to vanish in a point, say $h(s) = 0$. For $t$ very small, we have
$$ \lim_{t \rightarrow 0}\left| \int_{0}^{T}{h\left(s+\frac{tr}{T}\right) - h\left(s+\frac{tr}{T}-r\right)\d r} \right| = \left| \int_{0}^{T}{ h(s-r)\d r} \right| = 0$$
while
$$ \lim_{t \rightarrow 0}\int_{0}^{T}{\left| h\left(s+\frac{tr}{T}\right) - h\left(s+\frac{tr}{T}-r\right)\right| \d r} = \int_{0}^{T}{\left| h(s-r)\right| \d r},$$
proving that $h \equiv 0.$ 
\end{proof}

\section{A generalised form of Wirtinger's inequality}
The second part of this paper is devoted to establishing explicit bounds for a certain class of $\ell^p$-spaces. The idea of our approach goes back to the proof
that $TL\geq 2\pi$ in any Hilbert space which is based on an analogue of Wirtinger's inequality for Hilbert spaces. In the following we adapt
 this idea by using the work of Croce \& Dacorogna \cite{crocedac} who found the optimal constant in a generalised
set of Wirtinger inequalities, including the case of interest to us here. They showed that for
$$
u\in \left\{W^{1,p}_{\rm per}(0,1) \mbox{ with }\int_0^1 u(t)\,\d t=0  \mbox{ and } u(0)=u(1)\right\},
$$
where $W^{1,p}_{\rm per}$ is the space of $L^p$-functions $u$ whose weak first derivative lies in $L^p$, one has
$$
\left(\int_0^1|u(t)|^p\right)^{1/p}\leq C_p\left(\int_0^1|\dot{u}(t)|^p\,\d t\right)^{1/p},
$$
where
\begin{eqnarray}\label{Cpis}
C_p=\frac{p}{4(p-1)^{1/p}\int_0^1 t^{-\frac{1}{p}}(1-t)^{\frac{1}{p}-1}\,\d t}
\end{eqnarray}
is sharp. (Note that the integral appearing in the denominator is in fact the beta function $B(1/p',1/p)$ where $p'$ is the H\"older conjugate of $p$. Croce and Dacorogna consider functions defined on $(-1,1)$ but the form of the inequality here is more suitable for us in what follows.)

\begin{cor}
\label{WirtingersInequality}
Let $u\in W^{1,p}_{\rm per}([0,T], X)$ where $X$ is either $\ell^p(\mathbb{R}^n)$ or $L^p(M,\mu)$ and assume that $\int_0^T u(t)\,\d t=0$. Then
\begin{equation}\label{inLp}
\int_0^T\|u(t)\|^p_{X}\,\d t\leq C_p^p T^p\int_0^T\|\dot{u}(t)\|^p_{X}\,\d t,
\end{equation}
where $C_p$ is given in (\ref{Cpis}) and is optimal.
\end{cor}

\begin{proof}
By a simple change of variables it suffices to prove the result for $T=1$. When $X=\ell^p(\R^n)$ we have
\begin{eqnarray*}
    \int_{0}^1\sum_{j=1}^n|u_j(t)|^p\,\d t&=&\sum_{j=1}^n \int_{0}^1|u_j(t)|^p\,\d t\le C_p^p\sum_{j=1}^n\int_{0}^1|\dot u_j(t)|^p\,\d t,
\end{eqnarray*}
from which (\ref{inLp}) is immediate. One can see that the constant is optimal by considering $u=(u_1,\ldots,u_n)$ with $u_1\in W^{1,p}_{per}(0,1)$ and $u_j=0$ for $j=2,\ldots,n$.

Similarly, for $X=L^p(M,\mu)$ we have
\begin{eqnarray*}
\int_{0}^1\int_U|u(x,t)|^p\,\d\mu\,\d t&=&\int_U\int_{0}^1|u(x,t)|^p\,\d t\,\d\mu\\
&\le& C_p^p\int_U\int_{0}^1|\dot u(x,t)|^p\,\d t\,\d\mu=C_p^p\int_{0}^1\int_U|\dot u(x,t)|^p\,\d\mu\,\d t,
\end{eqnarray*}
and (\ref{inLp}) follows once more. Optimality of the constant follows by taking $f(t,x)=f(t){\bf 1}_A$ for some $f\in W^{1,p}_{per}(0,1)$ and $A\subset U$ with $\mu(A)>0$.
\end{proof}

\section{Improved lower bounds in $\ell^p(\mathbb{R}^n)$ and $L^p(M,\mu)$}
\label{Better Bounds for lp}

Having established Wirtinger's inequality for $W^{1,p}_{\rm per}([0,T],X)$ where $X$ is either $\ell^p(\mathbb{R}^n)$ or $L^p(M,\mu)$, we can now prove
 the second contribution of this paper. The simple proof is essentially that for $p=2$ due to \cite{busfishmar} which is a particular case of this result if one notes that $C_2^{-1}=2\pi$.

\begin{thm}\label{Minimal bound for L^p using Poincare}
Let $x$ be a non-constant $T$-periodic solution to $\dot{ x}=f( x)$ in either $X=\ell^p(\mathbb{R}^n)$ or $X=L^p(M,\mu)$. Further, suppose that $f$ is Lipschitz continuous from $X$ into $X$ with Lipschitz constant $L$. Then
\begin{equation}\label{TLC}
TL\ge C_p^{-1}.
\end{equation}
\end{thm}

\begin{proof}
As the function $x$ is a solution to the ODE, it is differentiable by definition. Moreover, a simple calculation shows that
$$
\int_0^Tx(t+h)-x(t)\,\d t=0.
$$
Hence Wirtinger's inequality for $W^{1,p}_{per}((0,T),X)$ is applicable to $x(t+h)-x(t)$ and thus
\begin{eqnarray*}
\int_0^T\| x(t+h)- x(t)\|^p_X\,\d t&\le& C_p^pT^p \int_0^T \|\dot{ x}(t+h)-\dot{ x}(t)\|^p_X\,\d t\\
&=& C_p^p T^p\int_0^T \|f( x(t+h))-f( x(t))\|^p_X\,\d t\\
&\le& L^pC_p^pT^p \int_0^T \| x(t+h)- x(t)\|^p_X\,\d t.
\end{eqnarray*}
Dividing both sides by $\int_0^T \| x(t+h)- x(t)\|^p_X\,\d t$, which is non-zero as $x$ is non-constant, yields (\ref{TLC}).   
\end{proof}

 Theorem \ref{Minimal bound for L^p using Poincare} gives an improved lower bound on the product of Lipschitz constant $L$ and period $T$ for the spaces $\ell^p(\mathbb{R}^n)$ and $L^p(M,\mu)$ for a range of $p$ around $p=2$. Figure \ref{fig:BestConstant from generalised Wirtinger} plots $C_p^{-1}$ against $p$ for $1\leq p\leq 4$, and shows that $C_p^{-1}>6$ for $1.43 \leq p\leq 3.35$.\\

\begin{figure}[h!]
  \begin{center}
\includegraphics[width=0.7\columnwidth]{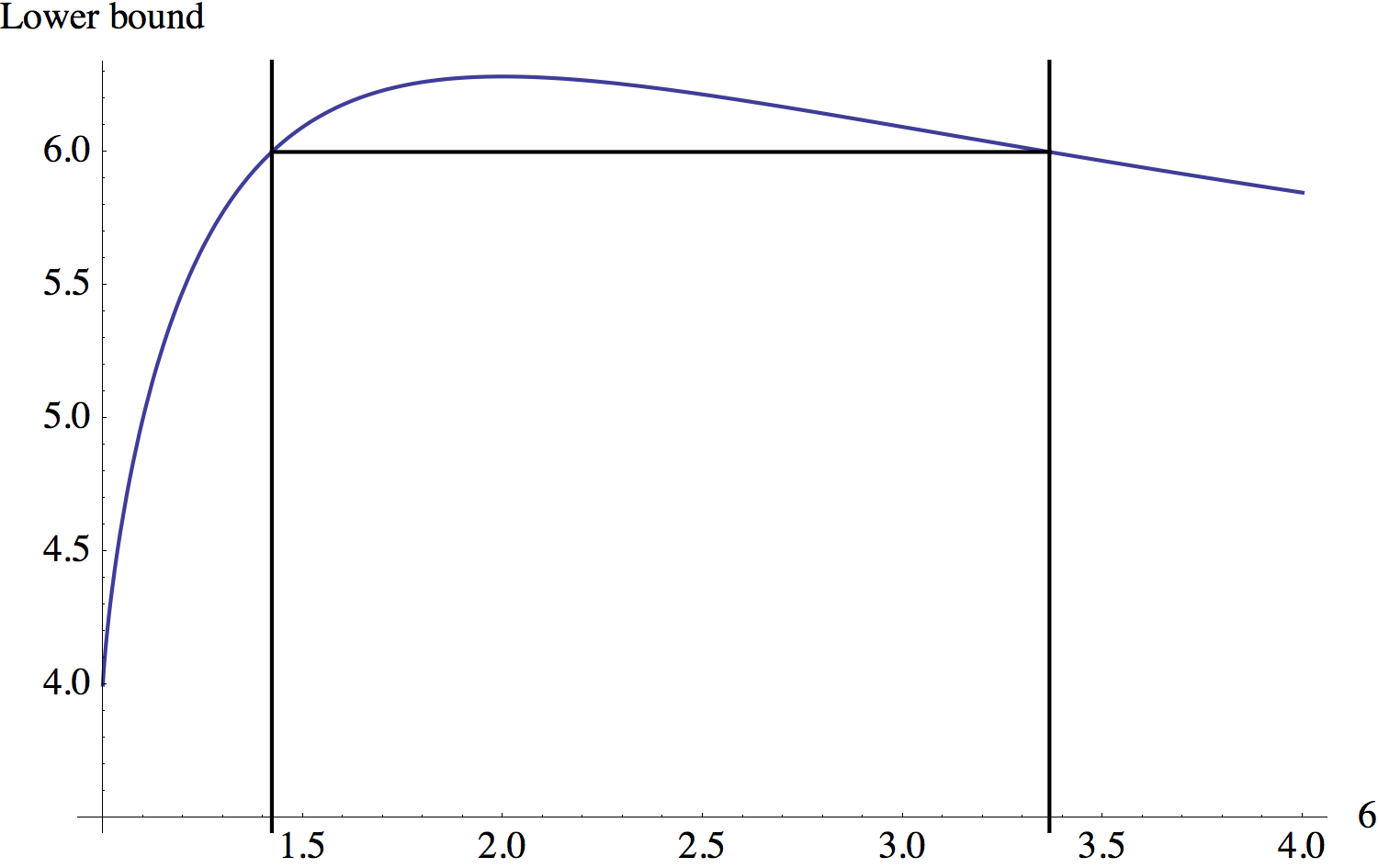}
\caption{Improved lower bound near $p=2$ using Wirtinger's inequality}
\label{fig:BestConstant from generalised Wirtinger}
\end{center}
\end{figure}

\textbf{Remark 1.} For $1\leq p<\infty$ one can construct an example of an ODE in $L^{p}(M,\mu)$ satisfying Lipschitz conditions on its derivative with period $2\pi$.
Suppose there are two sets $A\cap B= \emptyset$ such that
\[
0 < \mu(A) =\mu(B).
\]
and consider the ODE
$$\dot{z}=f(z)$$
with $f:L^{p}(M,\mu) \rightarrow L^{p}(M,\mu)$ given by
\[
f(z)= -\frac{\chi_{B}}{\mu(A)}\int_{A}{z d\mu} + \frac{\chi_{A}}{\mu(B)}\int_{B}{z d\mu}.
\]
Then H\"older's inequality gives that for $L=1$ the quantity
$$ I = \|f(z)-f(w)\|^p_{L^{p}(M,\mu)}$$
satisfies
\begin{eqnarray*}
\hspace{-1cm} I   \hspace{-0.2cm}&=&\hspace{-0.2cm} \left\| - \chi_B \frac{1}{\mu(A)}\int_{A}{z-w d\mu} + \chi_{A}\frac{1}{\mu(A)}\int_{B}{z-w d\mu} \right\|_{L^p}^p \\
\hspace{-0.2cm}&=&\hspace{-0.2cm} \left(\frac{1}{\mu(A)}\int_{A}{z-w d\mu}\right)^p \mu(B) + \left(\frac{1}{\mu(B)}\int_{B}{z-w d\mu}\right)^p \mu(A) \\
\hspace{-0.2cm}&\leq&\hspace{-0.2cm} \frac{1}{\mu(A)^p} \left(\int_{A}{|z-w|^p d\mu}\right) \mu(A)^{p-1}\mu(B) + \frac{1}{\mu(B)^p}\left(\int_{B}{|z-w|^p d\mu}\right) \mu(B)^{p-1}\mu(A) \\
\hspace{-0.2cm}&\leq&\hspace{-0.2cm} \|z-w\|_{L^p}^{p}
\end{eqnarray*}
and one explicit $2\pi-$periodic solution is given by
$$ z(t) = -(\cos{t})\chi_A + (\sin{t})\chi_B.$$
Notice that this example can be generalised further to the case when $0 < \mu(A) \neq \mu(B)$.\\

\textbf{Remark 2.} Let $X$ be a Banach space which obeys `almost' a Hilbert space structure in the sense of the norm, that is there exists a $\varepsilon>0$ such that
$$(1-\varepsilon) \| x \|_H \leq \| x \|_X \leq (1+\varepsilon) \| x \|_H.$$
Let $x:\mathbb{R}\rightarrow X$ be a $T$-periodic solution to the ODE $\dot{x}=f(x)$ with $f$ being Lipschitz continuous with Lipschitz constant $L$. Since \begin{eqnarray*}
\|f(x) - f(y)\|_H \leq \frac{1}{1-\varepsilon} \| f(x) - f(y) \|_X \leq \frac{1}{1 - \varepsilon} L \| x - y\|_X \leq \frac{1+\varepsilon}{1-\varepsilon} L \|x - y\|_H,
\end{eqnarray*}
it follows that $f$ is also Lipschitz continuous with respect to the Euclidean norm with Lipschitz constant $L'=\frac{1+\varepsilon}{1-\varepsilon}L$. At the same time,
the length of the curve $x$ as measured in the Hilbert space is smaller than $(1+\varepsilon)T$ and using the fact that $c=2\pi$ in any Hilbert space we may conclude
that $$TL\geq 2\pi\frac{1-\varepsilon}{(1+\varepsilon)^2}.$$
However, this approximation lags behind the numerical results for $\ell^p$ obtained at the beginning of this section, especially for high dimensions.\\

\textbf{Remark 3.} Dvoretzky's theorem in \cite{dvor} guarantees that for any $\varepsilon > 0$ there exists $n \in \mathbb{N}$ sufficiently large such that any Banach space
with $\dim X \geq n$ contains a two-dimensional subspace with Banach-Mazur distance to $\ell_2^2$ at most $1+\varepsilon$. The example of a simple circle in $\ell^2_2$ realizes
$TL = 2\pi$. This means that in any Banach space $X$ it is possible to construct an ODE satisfying $TL \leq 2\pi + \varepsilon$, where $\varepsilon$ depends only on
the dimension of $X$. We do not know whether there is always an ODE for which $TL \leq 2\pi$.


\section{Conclusion}
As discussed in the introduction, the key question is what intrinsic property of a space $X$ determines the largest (and hence best) constant $C_X$ such that $LT\ge C_X$. One of
 these intrinsic properties is strict convexity for which we have shown that the constant must be strictly larger than $6$. A natural question that arises is whether there exists
 a Banach space in which the optimal constant is neither $6$ nor $2\pi$.\\

However, explicit bounds are difficult to obtain. Even in the simple case $X=\ell^p(\R^n)$ this is not known, although our simple argument gives an explicit lower bound for
$p$ around $p=2$. It is interesting that a simple calculation shows that $C_p=C_{p'}$ when $p$ and $p'$ are conjugates; but it is not known whether the optimal constants
 in $\ell^p$ and $\ell^{p'}$ do in fact coincide (this interesting question was suggested to one of us in a personal communication from Mario Martelli).\\

While the use of an $L^p$-based Wirtinger inequality suits the $\ell^p$-spaces, there is no reason why these exponents should match. Given a Banach space $X$ it would be interesting to determine the optimal constants in the family of inequalities
$$
\left(\int_0^1 \| u(t)\|^p_X\,\d t\right)^{1/p}\leq C_p(X)\left(\int_0^T \|\dot{ u}(t)\|^p_X\,\d t\right)^{1/p},
$$
noting that as a consequence of such a family of inequalities and the argument of Theorem \ref{Minimal bound for L^p using Poincare} one would obtain
$$
TL\ge\sup_p C_p(X)^{-1}.
$$

\section*{Acknowledgement}
JCR is supported by an EPSRC Leadership Fellowship, grant number EP/G007470/1. MACN gratefully acknowledges funding from the European Research Council under the European
Union's Seventh Framework Programme (FP7/2007-2013) / ERC grant agreement n$^{\circ}$ 291053. SS is supported by a Hausdorff scholarship of the Bonn International
Graduate School.


\begin{thebibliography}{5}

\bibitem{busmar} S.N. Busenberg and M. Martelli, Bounds for the period of periodic orbits of dynamical systems. {\it Journal of Differential Equations} {\bf 67}, 359--371, 1987.

\bibitem{busfishmar}  S.N. Busenberg, D.C. Fisher, and M. Martelli, Better bounds for periodic solutions of differential equations in Banach spaces.
{\it Proceedings of the American Mathematical Society} {\bf 98}, 376--378, 1986.

\bibitem{busfishmar2} S.N. Busenberg, D.C. Fisher, and M. Martelli, Minimal periods of discrete and smooth orbits. {\it The American Mathematical Monthly} {\bf 96}, 5--17, 1989.

\bibitem{crocedac} G. Croce and B. Dacorogna, On a generalized Wirtinger inequality. {\it Discrete and Continuous Dynamical Systems} {\bf 9}, 1329--1341, 2003.

\bibitem{dvor} A. Dvoretzky, Some results on convex bodies and Banach spaces. Proc. Internat. Sympos. Linear Spaces (Jerusalem, 1960). Jerusalem: Jerusalem Academic Press. pp. 123--160, 1961.

\bibitem{lasyor} A. Lasota and J.A. Yorke, Bounds for periodic solutions of differential equations in Banach spaces. {\it Journal of Differential Equations} {\bf 10}, 83--91, 1971.

\bibitem{yor} J.A. Yorke, Periods of Periodic Solutions and the Lipschitz Constant. {\it Proceedings of the American Mathematical Society} {\bf 22}, 509--512, 1969

\bibitem{zev} A. A. Zevin, Sharp estimates for the amplitudes of periodic solutions to Lipschitz differential equations. {\it Dokl. Akad. Nauk.} {\bf 78}, 596-- 600, 2008

\bibitem{zev2} A. A. Zevin, Minimal periods of solutions of Lipschitzian differential equations in a vector space with an arbitrary norm. {\it Dokl. Akad. Nauk.} {\bf 444}, 602--604, 2012.


\end{thebibliography}
\end{document}